\documentclass[a4paper,8pt]{article}
\usepackage[francais]{babel}
\usepackage{geometry}
\usepackage[T1]{fontenc}
\usepackage[latin1]{inputenc}
\usepackage{amsmath,amssymb,mathrsfs,amsthm}
\usepackage{soul}
\usepackage{epsfig}
\usepackage{hyperref}
\usepackage{graphicx}
\usepackage{xypic}	
\usepackage{fancyhdr}
\fancypagestyle{plain}{
\lhead{}
\rhead{}

}

\newtheorem{theorem}{Th\'eor\`eme}[section]
\newtheorem{proposition}[theorem]{Proposition}
\newtheorem{lemma}[theorem]{Lemme}

\newtheorem{example}[theorem]{Exemple}
\newtheorem{remarque}[theorem]{Remarque}

\newcommand{\vc}{\|\cdot\|}

\newcommand{\Z}{\mathbb{Z}}

\newcommand{\p}{\mathbb{P}}

\newcommand{\B}{\mathcal{B}}
\newcommand{\A}{\mathcal{A}}
\newcommand{\R}{\mathbb{R}}
\newcommand{\Q}{\mathbb{Q}}
\newcommand{\eps}{\varepsilon}
\newcommand{\T}{\mathbb{T}}

\newcommand{\vf}{\varphi}
\newcommand{\si}{\sigma}
\newcommand{\Si}{\Sigma}
\newcommand{\C}{\mathbb{C}}
\newcommand{\ve}{\textbf{e}}
\newcommand{\N}{\mathbb{N}}
\newcommand{\lra}{\longrightarrow}
\newcommand{\e}{\textbf{e}}
\newcommand{\D}{\mathcal{O}(D)}
\newcommand{\cl}{\mathcal{C}^\infty}

\title{Sur le volume arithm\'etique sur les sch\'emas torique lisses}
\author{Mounir Hajli}
\date{}
\begin{document}

\maketitle

\begin{center}
{\sffamily Abstract}
\end{center}

Let $X$ be a smooth projective toric scheme over $\mathrm{Spec}(\Z)$. Let $\overline{D}$ be an equivariant
 line bundle on $X$, endowed with a continuous hermitian metric which is invariant by the action of the compact
 torus of $X(\C)$. We show that its arithmetic volume is given in terms of the Legendre-Fenchel transform associated
 to the metric. When $\overline{D}$ is supposed admissible, we characterize when it is arithmetically ample, nef or big
 in terms of combinatorial date.\\

\begin{abstract}
Soit $X$ un sch\'ema torique projective lisse sur $\mathrm{Spec}(\Z)$. Si $\overline{D}$
est un fibr\'e en droites \'equivariant sur $X$, muni d'une m\'etrique continue et invariante par
l'action du tore compact de $X(\C)$, nous montrons que son volume arithm\'etique s'exprime en fonction
de la
transform\'ee de Fenchel-Legendre associ\'ee \`a sa m\'etrique. Lorsque $\overline{D}$ est en plus admissible, nous
montrons que le fait que $\overline{D}$ soit arithm\'etiquement ample, nef, ou gros est caract\'eris\'e par
 des objets issus de la g\'eom\'etrie convexe.\\
\end{abstract}

 \begin{center}
 \large Introduction
 \end{center}

Soit $X$ une vari\'et\'e arithm\'etique de dimension relative
$d$ sur $\mathrm{Spec}(\Z)$, c'est \`a dire un sch\'ema projectif, int\`egre, plat sur $\Z$ et
de fibre g\'en\'erique lisse. Soit $\overline{L}$  un fibr\'e
en droites muni d'une m\'etrique hermitienne continue. On dit que $\overline{L}$ est admissible, si  $L$ est
relativement nef et
sa m\'etrique est  limite uniforme
d'une suite de m\'etriques de classe $\cl$ et semi-positives.
On consid\`ere les diff\'erentes notions de positivit\'e arithm\'etique suivantes:
\begin{enumerate}
\item $\overline{L}$ est \textit{ample} si le courant de Chern $c_1(\overline{L})$ est semi-positif sur $X(\C)$,
et pour tout $l$ assez grand, l'espace des sections globales $H^0(X,L^{\otimes{l}})$ est engendr\'e comme
un $\Z$-module par l'ensemble:
\[
\{s\in H^0(X,L^{\otimes l})|\, \|s\|_{\sup}<1 \}.
\]
 \item $\overline{L}$ est \textit{nef} si $L$ est relativement nef, le courant de Chern   $c_1(\overline{L})$
 est semi-positif sur $X(\C)$ et pour tout $P\in X(\overline{\Q})$ la hauteur de $P$ par rapport \`a $\overline{L}$
 est positive:
 \[
 h_{\overline{L}}(P)\geq 0\footnote{La hauteur d'un point $P\in X(\overline{\Q})$ par rapport \`a $\overline{L}$ est
 d\'efinie comme suit: Soit $K$ un corps de nombres dans lequel $P$ est d\'efini et $\mathcal{O}_K$ son anneau
 des entiers. Le point $P$ correspond \`a un unique morphisme de sch\'emas $\eps_P:\mathrm{Spec}(\mathcal{O}_K)
 \rightarrow X$. Soit $s$ un \'el\'ement non nul du $\mathcal{O}_K$-module $\eps_P^\ast L$. La hauteur de $P$
 est le r\'eel suivant:
 \[
 h_{\overline{L}}(P):=\frac{1}{[K:Q]}\bigl(\log \#(\eps_P^\ast L/(\mathcal{O}_K\cdot s))-\sum_{\si:K\rightarrow
 \C} \log \|s\|_{L,\si}(P)\bigr).
 \]}.
 \]
 \item $\overline{L}$ est
 \textit{gros} si  $L$ restreint  \`a la fibre g\'en\'erique de $X$ est gros et qu'il existe
 $l$ un entier positif non nul  et
  $s$ une section globale non nulle de $L^{\otimes l}$ tels que $\|s\|_{\overline{L}^{\otimes l}}(x)<1$
 pour tout $x\in X(\C)$.
\end{enumerate}
En plus, le volume arithm\'etique de $\overline{L}$ est d\'efini comme suit:
\[
 \widehat{\mathrm{vol}}(\overline{L})=\limsup_{l\mapsto \infty}\frac{\widehat{h}^0
 (X,\overline{L}^{\otimes l})}{l^{d+1}/(d+1)!}.
\]
o\`u $\widehat{h}^0(X,\overline{L}^{\otimes l}):=\log \#\widehat{H}^0(X,L^{\otimes l})$ et $\widehat{H}^0(X,
L^{\otimes l}):=
\bigl\{s\in
H^0(X,L^{\otimes l})\, |\, \|s\|_{\sup}\leq 1 \bigr\}$.
C'est un analogue arithm\'etique du volume g\'eom\'etrique pour les fibr\'es en droites sur une vari\'et\'e projective
d\'efinie sur un corps.\\

Dans cet article nous \'etudions les propri\'et\'es  ci-dessus dans le cadre de la g\'eom\'etrie torique.
 La g\'eom\'etrie arithm\'etique des vari\'et\'es toriques
 a \'et\'e \'etudi\'e de mani\`ere intense par Burgos, Moriwaki, Phillipon et Sombra  dans
 \cite{Burgos3}, \cite{Burgos2} et \cite{PS}. Comme dans le
cas g\'eom\'etrique, il s'av\`ere  qu'il  est possible de d\'ecrire certaines propri\'et\'es arithm\'etiques de ces
vari\'et\'es  en termes d'objets  issus de la g\'eom\'etrie convexe.\\

Commen\c{c}ons
tout d'abord par faire un bref rappel sur la construction des sch\'emas toriques.
 Soit $Q$ un $\Z$-module libre de rang fini et $P$ son dual. On consid\`ere un \'eventail $\Si$ sur $Q_\R=
 Q\otimes_\Z \R$ et on note par $X=X_\Si$  le sch\'ema torique sur $\mathrm{Spec}(\Z)$ associ\'e. On suppose
 en plus que $X$ est projectif et lisse. Le tore $\T=\mathrm{Spec}(\Z[P])$ (o\`u $\Z[P]$ est la $\Z$-alg\`ebre
 associ\'ee au group abelien $P$) s'identifie naturellement \`a un ouvert de $X$ et son action s'\'etend \`a
 une action sur $X$ entier.

 Soit $D$ un diviseur de Cartier \'equivariant sur $X$, c'est \`a dire, un diviseur de Cartier invariant par
 l'action du tore $\T$. Le diviseur $D$ correspond une "fonction support virtuelle" $\psi_D$  sur l'\'eventail
 $\Si$ qui permet de d\'efinir le polytope convexe:
 \[
 \Delta_D=\{x\in P_\R|\, <x,u>\geq \psi_D(u),\, \forall \, u\in Q_R\},
 \]
 o\`u $P_\R:=P\otimes_\Z \R$.

 Supposons que le fibr\'e en droites $\mathcal{O}(D)$ est muni d'une m\'etrique continue $\vc_{\overline{D}}$ invariante
 par l'action de $S_Q$ le
  tore compact de $\T(\C)$. On note par $\overline{D}=(D,\vc_{\overline{D}})$ le fibr\'e hermitien
 obtenu. Si $s_D$ d\'esigne la section rationnelle du fibr\'e en droites $\mathcal{O}(D)$ associ\'ee \`a $D$, on
 consid\`ere la fonction $g_{\overline{D}}:Q_\R\rightarrow \R$ d\'efinie comme suit:
 \[
 g_D(u):=\log \|s_D(\exp(-u))  \|_{\overline{D}},
 \]
 o\`u $\exp(-(\cdot)):Q_\R\lra X(\C)$ est l'application exponentielle associ\'ee.

On note par $\check{g}_{\overline{D}}
:P_\R\rightarrow [-\infty,+\infty[$ la transform\'ee de Legendre-Fenchel de $g_{\overline{D}}$, c'est \`a dire la
fonction d\'efinie pour tout $x\in P_\R$ par
\[
\check{g}_{\overline{D}}(x):=\inf_{x\in Q_\R}(<x,u>-g_{\overline{D}}(u)).
\]
  On montre que $\check{g}_{\overline{D}}(x)$ est finie si et seulement si $x\in \Delta_D$ et que $\check{g}_{\overline{D}}$ est concave sur $\Delta_D$.\\

\noindent{\large{\sffamily Principaux r\'esultats}}\\

Comme premier r\'esultat, nous \'etablissons que la g\'eom\'etrie torique nous fournit des exemples de
fibr\'es hermitiens nef mais qui ne sont pas gros:

\begin{proposition}[cf. proposition  \eqref{fibcanon}]
Soit $X$ une vari\'et\'e torique lisse de dimension relative $d$ sur $\mathrm{\mathrm{Spec}}(\Z)$. Soit
$\overline{D}_\infty$ un fibr\'e en droites \'equivariant engendr\'e par ses sections globales et
 muni de sa m\'etrique canonique. On a pour tout $l\in \N^\ast$
\[
\widehat{H}^0(X,l\overline{D}_\infty)=\bigl\{\pm \chi^m\, \bigl|\, m\in l\Delta_D\cap P  \bigr\}\cup\{0\}.
\]
En particulier, $l\overline{D}_\infty$  est nef mais il n'est pas gros.
\end{proposition}

En s'inspirant de \cite{Moriwaki}, nous \'etendons certaines de ses r\'esultats aux vari\'et\'es toriques lisses.
Cet article contient donc deux r\'esultats  majeurs: Th\'eor\`emes \eqref{x3} et \eqref{x4}. Nous \'etablissons une formule int\`egrale pour le
volume arithm\'etique de $\overline{D}$,
un fibr\'e en droites hermitien muni d'une m\'etrique continue et invariante par l'action du
tore compact, c'est l'objet du th\'eor\`eme  \eqref{x3}. Le th\'eor\`eme \eqref{x4} donne une interpr\'etation de la
positivit\'e arithm\'etique en termes de la g\'eom\'etrie convexe  lorsqu'on suppose en plus que $\overline{D}$ est
admissible.

\begin{theorem}\label{x3}\rm{[cf. th\'eor\`eme \eqref{volumetorpos}]}
 Soit $X$ une vari\'et\'e torique lisse. Soit $\overline{D}=(D,\vc_{\overline{D}})$ un fibr\'e en
droites  \'equivariant muni d'une m\'etrique    $\vc_{\overline{D}}$, continue  et
 invariante par l'action du tore compact de $X(\C)$. On a,
\[
 \widehat{\mathrm{vol}}(\overline{D})=(d+1)!\int_{\Theta_{\overline{D}}}\check{g}_{\overline{D}}(x)dx.
\]
o\`u
$\Theta_{\overline{D}}:=\{x\in \Delta_D|\, \check{g}_{\overline{D}}(x)\geq 0\}$.
\end{theorem}

Si l'on suppose en plus que $\overline{D}$ est admissible, alors nous d\'ecrivons les diff\'erents notions de la
 positivit\'e
arithm\'etique en termes de la combinatoire associ\'ee:
\begin{theorem}\label{x4}\rm{[cf. th\'eor\`eme \eqref{amplenefgros}]} Soit $X$ une vari\'et\'e torique lisse sur $\mathrm{Spec}(\Z)$ et $\overline{D}
=(D,\vc_{\overline{D}})$
un fibr\'e \'equivariant et admissible sur $X$ tel que  $h$ est
 invariante par l'action du tore compact de $X(\C)$. On a
\begin{enumerate}
\item $\overline{D}$ est ample si et seulement si  $\check{g}_{\overline{D}}(\e)>0 $, $\forall  \e\in \Delta_D\cap P$
    et $\psi_D$ est strictement concave.
\item $\overline{D}$ est nef   si et seulement si $\check{g}_{\overline{D}}(\e)\geq 0 $, $\forall \e\in \Delta_D\cap P$ et $\psi_D$ est concave.
\item  $\overline{D} $ est gros si et seulement si $g_{\overline{D}}(0)<0 $.
\end{enumerate}

\end{theorem}

Notre approche suit de pr\`es celle de Moriwaki dans  \cite{Moriwaki}. En effet,  \'etant donn\'e  un fibr\'e en
droites  \'equivariant $\overline{D}$ muni d'une m\'etrique   continue  et
 invariante par l'action du tore compact de $X(\C)$,  notre strat\'egie repose sur la
comparaison de la norme-sup avec la norme $L^2$ en utilisant  une variante faible
de l'in\'egalit\'e  Gromov (voir proposition \eqref{gro}) et sur le lemme \eqref{x2}. Notons que  Moriwaki utilise un
cas particulier de l'in\'egalit\'e de Gromov (voir \cite[lemme 1.4]{Moriwaki}) afin d'\'etablir
une formule int\'egrale pour le volume arithm\'etique (voir \cite[theorem 2.3 (1)]{Moriwaki}). Remarquons
que notre lemme \eqref{x2} peut \^etre vu comme une version faible du \cite[lemme 2.1]{Moriwaki}, \'etape aussi
cruciale dans la preuve du  \cite[theorem 2.3]{Moriwaki}.\\

 Dans \cite{Moriwaki}, Moriwaki consid\`ere le diviseur arithm\'etique $\overline{D}_{\bold{a}}$
sur $\p^n_\Z=\mathrm{Proj}(\Z[T_0,\ldots,T_n])$ et \'etudie ses propri\'et\'es arithm\'etiques. Ce fibr\'e hermitien est
d\'efini comme suit $\overline{D}_{\bold{a}}=(H_0, g_{\bold{a}})$ o\`u
$H_0=\{T_0=0\}$ et $g_{\bold{a}}(z)=\log(a_0+a_1|z_1|^2+\cdots+a_n|z_n|^2)$ pour tout $z\in \C^n$ avec
$a_0,a_1,\ldots,a_n$ sont des r\'eels strictement positifs, aussi
il introduit la fonction r\'eelle $\vf_{\bold{a}}$ sur $\R_{\geq 0}^{n+1}$ donn\'ee par
$\vf_{\bold{a}}(x_0,x_1,\ldots,x_n)=-\sum_{i=0}^nx_i \log x_i+\sum_{i=0}^n x_i\log a_i$ pour tout $x\in
 \R_{\geq 0}^{n+1} $.
Avec les notations de $\cite[\text{th\'eor\`eme}\; 2.3]{Moriwaki}$, on a $\overline{D}_\textbf{a}$ est
ample $($resp. nef$)$ si et seulement si $a_i>1$  (resp.
$a_i\geq 1$) pour tout $i=0,\ldots,n $.   Avec les notations de notre article, nous  v\'erifions que
$g_{\overline{D}_{\textbf{a}}}(u)=-g_{\bold{a}}(\exp(-u))$ pour tout $u\in \R^n$ et que
$\check{g}_{\overline{D}_{\textbf{a}}}(x_1,x_2,\ldots,x_n)=\vf_{\bold{a}}(1-\sum_{i=1}^nx_i,x_1,\ldots,x_n)$ pour
tout $(x_1,x_2,\ldots,x_n)\in \Delta_n$ (le polytope standard de $\R^n$). Ainsi, nous retrouvons le r\'esultat de
Moriwaki.\\

  Dans un travail r\'ecent de  Burgos,
  Moriwaki, Philippon et Sombra voir \cite{Burgos3}.
Ils  \'etablissent, sous des hypoth\`eses plus g\'en\'erales, deux r\'esultats analogues aux th\'eor\`emes \eqref{x3}
 et \eqref{x4},voir
 \cite[th\'eor\`emes 5.6, 6.1]{Burgos3}. Dans leur article, le sch\'ema torique peut \^etre singulier.
Mais  nous pensons que ce cas peut se d\'eduire du cas lisse \`a l'aide d'une r\'esolution de singularit\'es
\'equivariante et en utilisant le fait que le volume arithm\'etique est un invariant birationel.\\

Signalons au passage que les auteurs dans \cite{Burgos3} proc\'edent diff\'eremment (voir
\cite[remarque 5.7]{Burgos3}). Notre approche fournit une nouvelle preuve pour certains
r\'esultats du \cite{Burgos3} en particulier pour \cite[th\'eor\`emes 5.6, 6.1]{Burgos3}.
Observons que la propri\'et\'e de concavit\'e de la fonction $\psi_{\overline{D},v}$ dans
\cite[th\'eor\`eme 6.1]{Burgos3} est automatique puisque le fibr\'e hermitien est suppos\'e admissible. Notons que
 dans \cite{Burgos3} un fibr\'e admissible est dit semi-positif (voir \cite[d\'efinition 3.2]{Burgos3}).\\

\vspace{4cm}

\noindent{\large{\sffamily Organisation de l'article}}\\

La section \eqref{Chap1} est form\'ee de deux parties. La premi\`ere partie contient un survol de la g\'eom\'etrie
  des sch\'emas toriques, suivi d'un
r\'esultat  d\'ecrivant l'image  d'une vari\'et\'e torique
par un morphisme  \'equivariant  dans un espace projectif  (voir proposition \eqref{imagetoric}). Ce r\'esultat   sera utile pour
 la suite. La deux\`eme partie  regroupe
les d\'efinitions des diff\'erents objets et notions  qui seront \'etudi\'es dans ce texte. Dans cette deuxi\`eme
 partie,
nous d\'ecrivons  l'ensemble des sections petites d'un fibr\'e en droites  \'equivariant muni d'une m\'etrique
continue et invariante par l'action du tore compact. Pour cela, nous commen\c{c}ons par d\'eterminer l'ensemble des sections petites pour la m\'etrique $L^2$ pour
les diff\'erentes puissances du fibr\'e en droites hermitien en question, et \`a l'aide d'une variante faible de   l'in\'egalit\'e de Gromov nous
d\'eduisons celles de normes sup inf\'erieure \`a $1$.  C'est l'objet de la proposition \eqref{Base}.\\

Dans la section \eqref{positari}  nous \'etablissons les th\'eor\`emes \eqref{x3} et \eqref{x4}. Notre d\'emonstration
s'inspire de l'article  \cite{Moriwaki} et utilise de mani\`ere cruciale la variante faible de l'in\'egalit\'e
de Gromov (voir \eqref{Gromov11}).\\

\noindent{\large{\sffamily Remerciements}}\\

 Je tiens \`a remercier Vincent Maillot pour ses conseils et ses remarques autour de ce travail. Je tiens aussi \`a remercier le referee pour ses remarques pertinentes.

\tableofcontents

\section{Sur la g\'eom\'etrie des sch\'emas toriques}\label{Chap1}

Dans ce paragraphe, nous ferons un bref rappel sur la construction des sch\'emas toriques. On peut consulter
les r\'ef\'erences suivantes \cite{Demazure}, \cite{Fulton} et \cite{Oda} pour une introduction d\'etaill\'ee.\\

Soit $Q$ un $\Z$-module libre de rang $d$  et $P$ son $\Z$-module dual. On pose $Q_\R=Q\otimes_\Z \R$ et $P_\R=
P\otimes_\Z \R$.
 On appelle c\^one  strict dans $Q$  tout ensemble $\sigma\subseteq Q_\R$ de la forme $
\sigma=\Sigma_{i\in I} \R^+n_i$
qui ne contient aucune droite r\'eelle o\`u $\{n_i, i\in I\}$ est une famille finie d'\'el\'ements de $Q$. On
d\'efinit le dual $\si^\ast$ en posant $\si^\ast=\{v\in P_\R|\, <v,x>\geq 0,\, \forall x\in \si\}$. On dit
que $\tau\subset \si$ est une face de $\si$ si l'on peut trouver $v\in \si^\ast$ tel que $\tau=\si\cap \{v\}^\perp$.
 Un \'eventail de $Q_\R$ est une famille finie $\Sigma$ de c\^ones stricts de $Q$ tels que:
\begin{itemize}
 \item Si $\si\in \Sigma$ alors toute face $\tau $ de $\si$ appartient \`a $\Si$.
\item Si $\si$, $\si'\in \Si$, alors $\si\cap \si' $ est une face \`a la fois de $\si$ et de $\si'$.
\end{itemize}
La r\'eunion $|\Si|=\cup_{\si\in \Si}\si$ est appel\'ee le support de $\Si$. \\

Gr\^ace \`a Demazure \cite{Demazure}, on peut associer \`a tout \'eventail $\Si$ sur $Q$ un sch\'ema $\pi:X\rightarrow
\mathrm{Spec}(\Z)$. Ce sch\'ema est obtenu comme recollement d'une famille d'ouverts index\'ee par $\Si$ o\`u
chaque ouvert est le spectre de $\Z[P\cap \si^\ast]$ pour $\si\in \Si$ (o\`u $\Z[P\cap \si^\ast]$
 est la $\Z$-alg\`ebre associ\'ee au semi-groupe $P\cap \si^\ast$). Le tore $\T=\mathrm{Spec}(\Z[P])$ s'identifie
naturellement \`a un ouvert de $X$ dont son action s'\'etend \`a $X$ entier. On montre que $X$ est plat sur $\Z$,
normal, s\'epar\'e, de dimension absolue $d+1$, \`a fibres g\'eom\'etriquement in\`egres, voir \cite[\S 4 lemme 1
]{Demazure}. Dans la suite, toutes les vari\'et\'es toriques provenant d'un \'eventail seront suppos\'ees propres.\\

Soient $(Q,\Si) $ et $(Q',\Si')$ deux \'eventails. Un morphisme d'\'eventails $\vf:(Q',\Si')\rightarrow (Q,\Si) $ est un
morphisme de $\Z$-modules $\vf:Q'\rightarrow Q$ telle que l'application induite $\vf_\R: Q'_\R\rightarrow Q_\R$
d\'efinie par extension des scalaires
v\'erifie pour tout $\sigma'\in \Si'$, il existe $\si\in \Si$ tel que $\vf(\si')\subseteq \si$.
On note ${}^t \vf: P\rightarrow P'$ la transpos\'ee de $\vf$ et ${}^t\vf_\R$ l'application d\'efinie par extension
des scalaires.  La proposition suivante affirme qu'on peut recoller ces constructions locales pour obtenir un morphisme globale
\'equivariant $\vf_\ast$ et donne une condition n\'ecessaire et suffisante sur $\vf$ pour que $\vf_\ast$ soit propre:

\begin{proposition}\label{morphisme}
 Soit un morphisme d'\'eventails $\vf:(Q',\Si')\rightarrow (Q,\Si) $. Le morphisme de tores alg\'ebriques $
 \vf_\ast:\T'=\mathrm{Spec}(\Z[P'])\rightarrow \T=\mathrm{Spec}(\Z[P]),
$ induit par l'application duale ${}^t\vf_\R: P_\R\rightarrow P'_\R $, se prolonge en un morphisme:
\[
 \vf_\ast:X'\rightarrow X.
\]
Le morphisme $\vf_\ast$ est \'equivariant sous l'action de $\T'$ et $\T$. De plus, $\vf_\ast$ est propre si et seulement si $
 \vf_\R^{-1}(|\Sigma|)=|\Si'|)$.
\end{proposition}
\begin{proof}
 On peut consulter \cite[proposition. 1.13 et 1.15]{Oda}
\end{proof}

Un diviseur de Cartier \'equivariant $D$ sur $X$, est un diviseur de Cartier invariant par l'action du tore $\T$. \`A
$D$ on lui associe une  fonction  support virtuelle $\psi_D:Q_\R\rightarrow \R$ continue et lin\'eaire par
morceaux sur $\Si$, voir \cite[proposition 2.1]{Oda} et \cite[p. 66]{Fulton}. Cette fonction permet de d\'efinir le
polytope convexe suivant
\[
 \Delta_D=\{x\in P_\R|\, <x,u>\geq \psi_D(u),\, \forall \, u\in Q_R\},
 \]
 o\`u $P_\R:=P\otimes_\Z \R$.

La fonction $\psi_D$ et le polytope $\Delta_D$ codent plusieurs informations sur la positivit\'e de $D$. Par exemple,
le diviseur $D$ est g\'en\'er\'e par ses sections globales (resp. ample) si et seulement si $\psi_D$ est concave
(resp. strictement concave).

\begin{proposition}\label{x1}
Soit $D$  un diviseur de Cartier horizontal $\T$-invariant et $\D$ le faisceau inversible associ\'e \`a $D$.  Le $\Z$-module des sections globales de $\D$ est donn\'e par:
\[
 H^0(X,\D)=\bigoplus_{m\in \Delta_D\cap P}\Z\chi^m,
\]
o\`u $\chi^m$ est le caract\`ere associ\'e \`a $m$.
\end{proposition}
\begin{proof}
On peut consulter \cite[lemme 2.3]{Oda} et \cite[p.66]{Fulton}, les arguments donn\'es sur $\C$ s'\'etendent
imm\'ediatement \`a la situation sur $\mathrm{\mathrm{Spec}}(\Z)$.\\
\end{proof}

Soit $D$ un diviseur de Cartier \'equivariant sur $X$ et $\D$ le  fibr\'e en droites associ\'e qu'on suppose
 engendr\'e par ses sections globales. Il   d\'efinit alors
un morphisme \'equivariant de $X$
vers l'espace projectif de dimension $\#(\Delta_D\cap P)-1$. Nous allons d\'ecrire l'image de $X$ par
ce morphisme, c'est l'objet de la proposition \eqref{imagetoric}. Ce r\'esultat nous servira dans la suite pour \'etudier
le volume arithm\'etique associ\'ee \`a un fibr\'e en droites \'equivariant muni d'une m\'etrique continue et
invariante par l'action  de $S_Q$.

On note par $k$, un corps alg\'ebriquement clos. Soit $T^d=(k^{\times})^d$ le tore alg\'ebrique et
$\mathbb{P}^n(k)$ l'espace projectif  de dimension $d$ et $n$ respectivement. Soit
$\mathcal{A}=\{a_0,\ldots,a_n\}$ une suite de $n+1$ vecteurs de
$\mathbb{Z}^d$.

L'ensemble $\A$ d\'efinit une
action  de $T^d$ sur $\mathbb{P}^n(k)$:
\[
\ast_\mathcal{A}: T^d\times \mathbb{P}^n(k)\longrightarrow \mathbb{P}^n(k),\quad (s,x)\rightarrow
(s^{a_0}x_0:\cdots:s^{a_n}x_n).
\]
On note par $X_{\mathcal{A},1}$ l'adh\'erence de Zariski de l'image de l'application monomiale:
\begin{equation}\label{tor}
\ast_{\mathcal{A},1}:=\ast_\mathcal{A}|_{1}:\mathbb{T}^d\longrightarrow \mathbb{P}^n(k), \quad s\rightarrow
(s^{a_0}:\cdots:s^{a_n})
\end{equation}

C'est une vari\'et\'e torique projective au sens de Gelfand, Kapranov et Zelevinsky cf. \cite{GKZ}, c'est \`a dire
une sous-vari\'et\'e de $\mathbb{P}^n(k)$ stable par rapport \`a l'action de  $T^d$, avec une
orbite dense $X_{\mathcal{A},1}^\circ:=T^d\ast_\mathcal{A}1$.

\begin{proposition}\label{imagetoric}
Soit $X $ une vari\'et\'e torique de dimension $d$ provenant d'un \'eventail et $D$ un diviseur de Cartier \'equivariant.
Soit $\D$ le  fibr\'e en droites associ\'e  qu'on suppose engendr\'e par ses
sections globales. On d\'efinit un morphisme \'equivariant associ\'e \`a $D$, qu'on le note $\phi_D$, de la fa\c{c}on suivante:
\[
 \begin{split}
  \phi_D:X&\lra \mathbb{P}^{k_D}_\Z\\
x\,&\longmapsto (\chi^m(x) )_{m\in \Delta_D\cap P}
 \end{split}
\]
o\`u on a choisit un ordre sur les \'el\'ements $\Delta_D\cap P$, $k_D=\#(\Delta_D\cap P)-1$.

Alors il existe $\B$ un sous ensemble fini de vecteurs de $\Z^d$, tel que l'image de $X(k)$ par $\phi_D$
co\"incide avec $X_{\B,1}$.
\end{proposition}

\begin{proof}
D'apr\`es \eqref{morphisme}, $\phi_D$ provient d'un morphisme d'\'eventails \[ \phi:(\Z^d,\Si_X)\lra (\Z^n,\Si_{\p^n}) \] o\`u $\Si_X$ est l'\'eventail de $X$ et $\Si_{\p^n}$
celui de $\mathbb{P}^n_\Z$.\\

On pose $c_k=\vf(e_k^{(d)})$ pour $k=1,\ldots,d$ (o\`u $e_k^{(d)}$ d\'esigne le $k$-\`eme vecteur de la base usuelle
 de $\Z^d$) et on note par $B$ la matrice d'ordre $n\times d$ qui a pour colonnes $c_1,c_2,\ldots,c_d$. Si l'on
note par $b_j$ avec  $j=1,\ldots,n$  les lignes de $B$ alors ${}^t \vf $ (le morphisme  dual de $\vf$) s'\'ecrit
\[\begin{split}
 {}^t\phi:\Z^n&\lra \Z^d\\
m&\longmapsto {}^t B\cdot m,
\end{split}
\]
o\`u ${}^t B $ est la matrice transpos\'ee de $B$. Le morphisme ${}^t\vf$ induit l'homomorphisme de $\Z$-alg\`ebres suivant:
\begin{equation}\label{truc}
 \begin{split}
 \Z[\Z^n]&\lra \Z[\Z^d]\\
\chi^m&\longmapsto \chi^{{}^t \phi(m)}.
\end{split}
\end{equation}
Comme  \[
       \begin{split}
        {}^t \phi(e_j^{(n)})&={}^t B\cdot e_j^{(n)}={\sum_{j=1}^d b_{k,j}e^{(d)}_j },
 \end{split}
      \]
alors \eqref{truc} devient
\begin{equation}
 \begin{split}
 \Z[\Z^n]&\lra \Z[\Z^d]\\
\chi^{e_j^{(n)}}&\longmapsto \chi^{{\sum_{j=1}^d b_{k,j}e^{(d)}_j }}
\end{split}
\end{equation}




Soit $s\in \mathrm{\mathrm{Spec}}(\Z[\Z^d])(k)=(k^{\times})^d $, on a \[s=(s_1,\ldots,s_d)=(X^{e^{(d)}_1}(s),\ldots,X^{e^{(d)}_d}(s) ).\]
Ce point s'envoie par $\phi_D$ sur
\begin{align*}
\bigl(1,X^{\sum_{j=1}^d b_{1,j}e^{(d)}_j }(s),\ldots, X^{\sum_{j=1}^d b_{n,j}e^{(d)}_j }(s)
\bigr)&=\bigl(1,\prod_{j=1}^d X^{b_{1,j}e^{(d)}_j
}(s),\ldots,\prod_{j=1}^d X^{b_{n,j}e^{(d)}_j }(s) \bigr)= \bigl(1,s^{b_1},\ldots,s^{b_n}\bigr).
   \end{align*}
(puisque $X^{b_{k,j}e^{(d)}_j }(s)=s_j^{b_{k,j}}$ pour $j=1,\ldots,d$).\\

On conclut que le morphisme $\phi_D$ co\"incide sur $(k^{\times})^d$ avec:
\begin{align*}
(k^\times)^d &\longrightarrow \mathbb{P}^n(k)\\
s&\longrightarrow (1,s^{b_1},\ldots,s^{b_n}).
\end{align*}

On termine la d\'emonstration  en rappelant  que $\chi^{{}^t\phi(e_j^{(n)}) } $ sont les sections globales de $\D$ qui correspondent \`a
l'ensemble  $\Delta_D\cap P$ (cf. \eqref{x1}).
\end{proof}


 On suppose que
 $X$ est projective et lisse dans la suite. Soit $D$ un diviseur de Cartier \'equivariant sur $X$ et on suppose
  que $\D$ est engendr\'e par ses sections globales sur $X$. Soient
$\psi_D$ la fonction support et $\Delta_D$ le polytope associ\'es \`a $D$.

On munit $\D(\C)$ d'une m\'etrique hermitienne continue $\vc_{\overline{D}}$, qu'on suppose invariante par l'action du
sous-tore compact $S_Q:=\{t\in \T(\C)\,|\, |t|=1\}$. Dans la suite, on notera le fibr\'e hermitien
$(D,\vc_{\overline{D}})$
par   $\overline{D}$.\\


Le quotient de $X(\C)$ par le sous-tore compact $S_Q$ est la vari\'et\'e \`a
coins  associ\'ee, not\'ee $X_{\R_{\geq 0}} $. On montre que
$\mathrm{Hom}(P,\R_{\geq 0})$  (l'ensemble des morphismes de semi-groupe avec \'el\'ement neutre de $P$ vers
$(\R_{\geq 0},\times)$) s'identifie un  ouvert dense de $X_{\R_{\geq 0}} $ not\'ee $X^\circ_{\R_{\geq 0}}$
 (voir \cite[\S 4]{Fulton} pour la construction). Comme   $Q_\R\simeq \mathrm{Hom}(P,\R)$ alors on a la
param\'etrisation suivante donn\'ee par l'exponentielle usuelle:
\[\begin{split}
   Q_\R&\longrightarrow X^\circ_{\R_{\geq 0}}\\
u&\longmapsto \exp(-u).
  \end{split}
\]

Soit $s_D$ la  section rationnelle  \'equivariante de $\D$ associ\'ee \`a $D$.  On pose
\[
 g_{\overline{D}}(u):=\log \|s_D (\exp(-u) )\|_{\overline{D}}\quad\forall u\in Q_\R,
\]
et on note par $\check{g}_{\overline{D}} $, la transform\'ee de Legendre-Fenchel de $g_{\overline{D}}$, qui  est par d\'efinition:
\[
 \check{g}_{\overline{D}}(x):=\inf_{u\in Q_\R}\bigl(\left<x,u\right>-g_{\overline{D}}(u)\bigr),\quad \forall x\in P_\R.
\]
On montre que $\check{g}_{\overline{D}} $ est finie si et seulement si $x\in \Delta_D$ et qu'elle est concave sur cet ensemble. On pose
\begin{equation}\label{theta}
 \Theta_{\overline{D}}:=\bigl\{x\in \Delta_D\, |\, \check{g}_{\overline{D}}(x)\geq 0 \bigr\}.
\end{equation}

\vspace{1cm}

 On munit $X(\C)$ d'une forme volume  $\Omega$,
de classe $\cl$, invariante par l'action de $S_Q$  et
telle que $\int_X \Omega=1$\footnote{On peut construire $\Omega$ de la mani\`ere
suivante: Soit $A$ un
fibr\'e en droites \'equivariant et tr\`es ample sur $X(\C)$. On  consid\`ere
  $h_0$ l'image r\'eciproque de la m\'etrique
de Fubini-Study par  le morphisme \'equivariant d\'efini par
$A$, c'est \`a dire  $X(\C)\lra \p^{\dim H^0(X,A)},x\longmapsto (\chi^m(x))_{m\in
\Delta_A\cap M}$, alors on v\'erifie que $h_0$  est une m\'etrique hermitienne sur $A$, de classe $\cl$,
 d\'efinie positive et invariante par l'action
 du tore compact de $X(\C)$. On pose alors $\Omega:=c_1(A,h_0)^d/\int_X c_1(A,h_0)^d$.}. Soit $h_{\overline{D}}$ une  m\'etrique hermitienne continue  sur $\D(\C)$. On pose
$\overline{D}:=(D,\vc_{\overline{D}})$. Pour $s,t\in H^0(X,\D)$, on pose
\[
 \left<s,t\right>_{\overline{D},\Omega}=\int_{X(\C)}h_{\overline{D}}(s,t)\,\Omega \quad \text{et}\quad \|s\|_{L^2,\overline{D},\Omega}:=\sqrt{\left<s,s\right>
}_{\overline{D},\Omega}.
\]
On peut aussi munir $H^0(X,\D)$  de la norme-sup:
\[
\|s\|_{\overline{D},\sup} := \sup_{x\in X(\C)}\|s\|_{\overline{D}}(x),\quad\text{pour}\; s\in H^0(X,\D).
\]

On pose
\[
 \widehat{H}^0_{L^2}(X,\overline{D}):=\bigl\{s\in H^0(X,\D)\, \bigl|\; \|s\|_{L^2,\overline{D},\Omega}\leq 1 \bigr\},
\]
et
\[
 \widehat{H}^0(X,\overline{D}):=\bigl\{s\in {H}^0(X,\D)\,\bigl| \; \|s\|_{\overline{D},\sup}\leq
1\bigr\},
\]
$\widehat{H}^0(X,\overline{D})$ est appel\'e l'ensemble des  sections petites de $\overline{D}$. Notons que
$\widehat{H}^0(X,\overline{D})\subseteq \widehat{H}^0_{L^2}(X,\overline{D})$. On note par
  $\left<\widehat{H}^0(X,l\overline{D} )\right>_\Z$
   le $\Z$-module engendr\'e par   $\widehat{H}^0(X,l\overline{D})$ et par
   $\left<\{s\in H^0(X,\mathcal{O}(lD))
   \,|\, \|s\|_{l\overline{D},\sup}<1\}\right>_\Z$  le $\Z$-module engendr\'e par   $\{s\in H^0(X,\mathcal{O}(lD))\,|
   \, \|s\|_{l\overline{D},\sup}<1\}$.\\

D'apr\`es \cite{Zhang}, \cite{Maillot},  on dit que $\overline{D}$ est admissible si $\D$ est relativement nef et
sa  m\'etrique   $\vc$ est limite uniforme
d'une suite de m\'etriques $(\vc_k)_{k\in \N}$  de classe $\cl$ et semi-positives sur $\D(\C)$.
\begin{remarque}\rm{
Notons  que la notion de m\'etrique admissible correspond \`a la notion de m\'etrique semi-positive consid\'er\'ee
par Burgos, Moriwaki, Philippon et Sombra,
voir \cite[d\'efinition 1.4.1]{Burgos2} et \cite[p.15]{Burgos3}.

 }
\end{remarque}

Dans \cite{Moriwaki2} Moriwaki introduit trois notions de positivit\'e arithm\'etique. La g\'eom\'etrie d'Arakelov d\'evelopp\'ee
dans \cite{Maillot} permet d'\'etendre ces  trois notions de positivit\'e
 aux fibr\'es en droites admissibles. Plus pr\'ecis\'ement,
soit $X$ une vari\'et\'e arithm\'etique projective et $\overline{D}$ un fibr\'e en droites hermitien
muni d'une m\'etrique continue  sur $X$. On consid\`ere les diff\'erentes notions de positivit\'e arithm\'etique suivantes:

\begin{enumerate}
\item $\overline{D}$ est \textit{ample} si le courant de Chern $c_1(\overline{D})$ est semi-positif sur $X(\C)$,
et pour tout $l$ assez grand, l'espace des sections globales $H^0(X,\mathcal{O}(lD))$ est engendr\'e comme
un $\Z$-module par l'ensemble:
\[
\{s\in H^0(X,l\overline{D})|\, \|s\|_{\sup}<1 \}.
\]
 \item $\overline{D}$ est \textit{nef} si $\D$ est relativement nef, le courant de Chern   $c_1(\overline{D})$
 est semi-positif sur $X(\C)$ et pour tout $P\in X(\overline{\Q})$ la hauteur de $P$ par rapport \`a $\overline{D}$
 est positive:
 \[
 h_{\overline{D}}(P)\geq 0.
 \]
 \item $\overline{D}$ est
 \textit{gros} si  $\D$ restreint  \`a la fibre g\'en\'erique de $X$ est gros et qu'il existe
 $l$ un entier positif non nul  et
  $s$ une section globale non nulle de $\mathcal{O}(lD)$ tels que $\|s\|_{l\overline{D}}(x)<1$
 pour tout $x\in X(\C)$.
\end{enumerate}

Un exemple int\'eressant de m\'etrique admissible est celui de m\'etrique canonique sur un fibr\'e en droites
\'equivariant $\D$  au-dessus d'une vari\'et\'e torique projective non-singuli\`ere $X$. On \'etablit qu'on peut
associer \`a $D$, de mani\'ere canonique, une m\'etrique continue not\'ee par $\vc_{D,\infty}$,   d\'ecrite uniquement
par la combinatoire de la vari\'et\'e $X$. En plus, on  peut montrer que $\vc_{D,\infty}$ est admissible lorsque $\D$
est engendr\'e par ses sections globales.
Il existe trois constructions \'equivalentes: La construction due \`a Batyrev
et Tschinkel \cite[\S 2.1]{Batyrev}, celle de Zhang \cite[th\'eo\`eme 2.2]{Zhang} et la construction par image inverse,
voir par exemple \cite[\S 3.3]{Maillot}.
On va d\'ecrire ici la trois\`eme construction qui nous sera utile pour la suite: Soit $D$ un diviseur de Cartier
 \'equivariant et $\D$ le fibr\'e en
droites associ\'e qu'on suppose engendr\'e par ses sections globales. On d\'efinit un morphisme  \'equivariant $\phi_D$
associ\'e de la fa\c{c}on suivante:
\begin{align*}
\phi_D:X &\longrightarrow \mathbb{P}^{k_D}\\
x&\longmapsto (\chi^m(x))_{m\in \Delta_D\cap P}
\end{align*}
o\`u $k_D=\#(\Delta_D\cap P)-1$.  On note par $\overline{\mathcal{O}(1)}_\infty$ le fibr\'e de Serre sur $\p^{k_D}$
muni de la m\'etrique d\'efinie pour toute section m\'eromorphe de $\mathcal{O}(1)$ par:
\[
\|s(x)\|_\infty=\frac{|s(x)|}{\sup_{0\leq i\leq k_D}|x_i|}.
\]
Cette m\'etrique est la m\'etrique de Batyrev-Tschinkel ou de Zhang pour le faisceau $\mathcal{O}(1)$ sur $\p^{k_D}$
consid\'er\'ee comme vari\'et\'e torique. En posant $\|\cdot\|_{D,\infty}=\phi_D^\ast \|\cdot\|_\infty$, alors on
\'etablit que cette m\'etrique est la m\'etrique de Batyrev-Tschinkel ou de Zhang pour le faisceau $\D$ sur $X$, voir
\cite[th\'eor\`eme 3.3.10]{Maillot}.

Ces fibr\'es hermitiens fournissent des exemples de fibr\'es hermitiens nef qui ne sont pas gros sur tout
sch\'ema torique lisse, comme le montre la proposition suivante:
\begin{proposition}\label{fibcanon}
Soit $X$ une vari\'et\'e torique lisse de dimension relative $d$ sur $\mathrm{\mathrm{Spec}}(\Z)$. Soit
$\overline{D}_\infty$ un fibr\'e en droites \'equivariant engendr\'e par ses sections globales et
 muni de sa m\'etrique canonique. On a pour tout $l\in \N^\ast$
\begin{equation}\label{JKLP}
\widehat{H}^0(X,l\overline{D}_\infty)=\bigl\{\pm \chi^m\, \bigl|\, m\in l\Delta_D\cap \Z^d  \bigr\}\cup\{0\}.
\end{equation}
En particulier, $l\overline{D}_\infty$ est nef mais il n'est pas gros.
\end{proposition}

\begin{proof} On note par $S_Q$ le tore compact de $X(\C)$.
 On  fixe
 un entier positif non nul $l$
  et soit $s\in \widehat{H}^0(X,\overline{D}_\infty)\setminus \{0\}
$. En particulier, on a $\|s\|_{D,\infty}(x)=|s(x)|\leq 1$ pour tout $x\in
S_Q$\footnote{
Si l'on consid\`ere  le morphsime $\phi_D$ d\'efini
par $D$:
\begin{align*}
\phi_D:X &\longrightarrow \mathbb{P}^{k_D}\\
x&\longmapsto (\chi^m(x))_{m\in \Delta_D\cap P}.
\end{align*}
Soit $s\in \widehat{H}^0(X,l\overline{D}_\infty)$.  Comme $\vc_{D,\infty}=\phi^\ast \vc_\infty$ et que $\phi_D$ envoie
le tore compact $X$ sur celui de $\p^{k_D}$, alors $\|s(x)\|_{D,\infty}=|s(x)|$ pour tout $x\in S_Q$ et $s\in
H^0(X,l\overline{D}_\infty)$.}et $s$ n'est pas identiquement nul sur $S_Q$ (Cela r\'esulte du fait que
le tore compact $S_Q$ est Zariski-dense dans $X(\C)$). Par cons\'equent l'int\'egrale suivante est
finie et elle est n\'egative:
\[
M(s):=\int_{S_Q}\log |s(x)|d\mu\leq 0.
\]
($d\mu$ d\'esigne la mesure de Haar normalis\'ee sur $S_Q$).
D'apr\`es \cite[\S 7.3.]{Maillot} la hauteur canonique $h_{\overline{D}_\infty}(\mathrm{div}(s))$ du
cycle $\mathrm{div}(s)$ est donn\'ee par la formule suivante:
\[
\begin{split}
h_{\overline{D}_\infty}(\mathrm{div}(s))&=\deg(D)\,M(s).\\
\end{split}
\]

 Or,  cette quantit\'e est positive d'apr\`es
\cite[proposition 5.5.7]{Maillot}. On   d\'eduit que

\[
M(s)=0.
\]
Par l'in\'egalit\'e de Jensen et comme $|s(x)|\leq 1$ pour tout $x\in
S_Q$ on a
\[0=\int_{S_Q}\log |s(x)|d\mu\leq \log
\bigl(\int_{S_Q}
|s(x)|d\mu\bigr)\leq 0.\]
  Par continuit\'e de $|s|$, on
obtient que $ |s|=1$ sur $S_Q$. Si l'on \'ecrit  $s=\sum_{m \in l\Delta_D\cap P} a_m \chi^m$, avec $a_m$ sont des entiers pour tout
$m  \in l\Delta_D\cap P$, alors
 $1=|s(x)|^2=\sum_{m,m'  \in l\Delta_D\cap P}a_m a_{m'} \chi^{m-m'}(x) $ pour tout $x \in S_Q$.
 En int\'egrant sur $S_Q$ on obtient
\[
\sum_{m \in l\Delta_D\cap P} a_m^2 =1.
\] Comme tous les $a_\nu$ sont des entiers, on en d\'eduit qu'il existe $m_0 \in l\Delta_D\cap P$, tel que $a_m=0$ si $m\neq m_0$ et $|a_{m_0}|=1$. Par suite,
\[
s=\pm \chi^{m_0} \text{et} \quad \|s\|_{\infty,\sup}=1.
\]
Donc $\overline{D}_\infty$ n'est pas gros, et on a:
\[
 \widehat{H}^0(X,\overline{D}_\infty)=\bigl\{\pm \chi^m\,\bigl|\, m \in
l\Delta_{D}\cap P  \bigr\}\cup\{0\}.
\]

De cette \'egalit\'e et puisque $\overline{D}_\infty$ est admissible  on conclut \`a l'aide de \cite[proposition
5.5.7]{Maillot} que  $\overline{D}_\infty$ est nef.
\end{proof}

Dans la suite, nous allons \'etablir un r\'esultat qui g\'en\'eralise la proposition \eqref{fibcanon}. Plus
pr\'ecis\'ement,
 nous allons d\'ecrire l'ensemble des sections petites d'un fibr\'e en droites
hermitien muni d'une m\'etrique hermitienne continue et invariante par le tore compact de la
vari\'et\'e torique,
en fonction de la transform\'ee de Fenchel-Legendre associ\'ee. C'est l'objet de la proposition
\eqref{Base}. La preuve  suivra le raisonnement fait dans \cite[proposition 1.5]{Moriwaki}
  et nous utiliserons de mani\`ere cruciale
une variante faible de l'in\'egalit\'e de Gromov (voir proposition \eqref{gro}).\\

\noindent{\large{\sffamily L'ensemble des sections petites}}\\

Lorsqu'on est
 dans la situation o\`u la m\'etrique est de classe $\cl$ alors un moyen pratique pour
le calcul du volume arithm\'etique consiste \`a comparer la norm-sup avec la norme $L^2$ moyennant l'in\'egalit\'e
de Gromov, et d'utiliser le fait que la norme $L^2$ est hermitienne, voir par exemple \cite{Amplitude}. Malheureusement
cette in\'egalit\'e n'est plus valable lorsqu'on suppose que la m\'etrique est uniquement
 continue.

 Dans \cite{Burgos3}, les auteurs
\'evitent le recours \`a cette technique,  ils  montrent
que la base des sections toriques est orthogonale pour la
norme-sup, voir \cite[proposition 5.2]{Burgos3} et \cite[remarque 5.7]{Burgos3}.

  Notre approche ici repose sur une variante faible de l'in\'egalit\'e de Gromov (voir proposition \eqref{gro}).
  Nous utilisons cette in\'egalit\'e pour comparer la norm-sup avec la norme $L^2$, et nous verrons que
cela est suffisant pour \'etablir une formule int\'egrale pour le volume arithm\'etique.

\begin{proposition}\label{gro} Soit $Y$ une vari\'et\'e diff\'erentielle compacte complexe munie de $\Omega$, une forme
volume de classe $\cl$.  Soit $\overline{L}$ un fibr\'e en droites holomorphe muni d'une m\'etrique hermitienne continue.
On a, pour tout $\eps>0$ il existe une constante $C>0$ telle que pour tout  $k>0$ et pour toute section holomorphe $s$ de $k L$, on ait
\begin{equation}
\|s\|_{k\overline{L},\sup}\leq C e^{\eps\, k} \|s\|_{L^2,k\overline{L},\Omega}.
\end{equation}

\end{proposition}

\begin{proof}
Voir par exemple  \cite[lemme 3.2]{BermanBoucksom}.
\end{proof}
Comme premi\`ere application de cette in\'egalit\'e, nous allons
 d\'ecrire l'ensemble des  sections petites de $\overline{D}$:
\begin{proposition}\label{Base} Soit $\overline{D}$ un fibr\'e en droites muni d'une m\'etrique continue et invariante
par l'action de  $S_Q$. On fixe $l$, un entier positif non nul. On a:
\begin{enumerate}
\item $\widehat{H}^0(X,l\overline{D} )\neq \{0\}$ si et seulement si $l\Theta_{\overline{D}}\cap P \neq
    \emptyset$.
\item Si $l\Theta_{\overline{D}}\cap P \neq \emptyset$, alors $\left<\widehat{H}^0(X,l\overline{D} )
    \right>_\Z=\bigoplus_{\e\in l \Theta_{\overline{D}}\cap P} \Z \chi^\e$.
\end{enumerate}

\end{proposition}


\begin{lemma}\label{convextreme111}
 Soit $\phi\in \widehat{H}^0_{L^2}(X,l\overline{D} )$, si l'on \'ecrit
\[                                                            \phi=\sum_{\ve\in l\Delta_D\cap P} c_\ve
\chi^\ve,
                                                      \]
alors $\{\ve\, |\, c_\ve\neq 0\}  \subset l \Theta_{\overline{D}}$.
\end{lemma}

\begin{proof}

On peut supposer $\phi\neq 0$, Posons $\{\ve\, |\, c_\ve\neq 0\}=\{\ve_1,\ldots,\ve_m \}$, avec $\ve_i\neq \ve_j$ , si $i\neq j$. Soit $\textbf{e}_i$ un point extr\'emal de $\mathrm{Conv}\{\ve_1,\ldots,\ve_m \}$.  Montrons que $\ve_i\in l\Theta_{\overline{D}}$. On peut supposer que $i=1$.\\

On a, pour tout $k\in \N_{\geq 1}$,
\[
 \phi^k= c^k_{\ve_1} \chi^{k \ve_1}+ \sum_{ \substack{k_1,\ldots,k_m \in \Z_{\geq 0}\\
k_1+\cdots+k_m=k,k_1\neq k}
 }\frac{k!}{k_1!\cdots k_m !} c_{\ve_1}^{k_1}\cdots c_{\ve_m}^{k_m} \chi^{k_1 \ve_1} \cdots \chi^{k_m
\ve_m}.
\]

On v\'erifie que $k \ve_1\neq k_1 e_1+\cdots + k_m \ve_m  $, pour tout $k_1,\ldots,k_m \in \Z_{\geq 0} $ tels que $k_1+\cdots+k_m=k $ et $k_1\neq k $. Sinon, $\ve_1=(\frac{k_2}{k-k_1})\ve_2 +\cdots+(\frac{k_m}{k-k_1})\ve_m$. Cela contredit le fait que $\ve_1$ est un point extr\'emal de $Conv(\ve_1,\ldots,\ve_m ) $, par suite on peut \'ecrire

\[
 \phi^k=c^k_{\ve_1} \chi^{k \ve_1}+ \sum_{\ve'\in \Z^d_{\geq 0}, \ve'\neq k\ve_1} c'_{\ve'}\chi^{\ve'},
\]cela implique que
\[
\left<\phi^k,\phi^k\right>_{kl\overline{D},\Omega} =c^{2k}_{\ve_1}\left<\chi^{k \ve_1},\chi^{k \ve_1}\right>_{kl\overline{D},\Omega}+(\text{r\'eel positif }),
\]

(l'invariance de $\Omega$ et de $h$ par l'action du tore compact de $X(\C)$, implique que
$\left<\chi^{\ve_1}\cdots\chi^{\ve_r},\chi^{\ve'_1}\cdots\chi^{\ve'_r}\right>_{k l\overline{D},\Omega}=0,$
d\`es que $(\ve_1,\ldots,\ve_r )\neq (\ve'_1,\ldots,\ve'_r )$).\\

 On en  d\'eduit que\[
\left<\chi^{k \ve_1},\chi^{k \ve_1}\right>_{kl\overline{D},\Omega}\leq 1,\quad (\text{rappelons que}\;c^{2k}_{\ve_1}\in \Z).
\]

Notons par $s_0,\ldots,s_{n} $ les sommets du polytope $\Delta_D$, o\`u $n=\#(\Delta_D\cap P)-1$.\\

 D'apr\`es \eqref{imagetoric}, il existe $\B=\{b_0,\ldots,b_{n}\}$ un sous-ensemble de $\Z^d$ tel que
(Quitte \`a
r\'eordonner les indices) on a:
 \[
 \chi^{s_i}(t)=t^{b_i}\quad \forall\, t\in \T_P\quad \forall\,  i=0,\ldots,n.
 \]
On peut supposer que $e_0=0$ et $s_0=0$. \\

Par hypoth\`ese $\ve_1\in (l\Delta_D)\cap P$.  Il existe donc des rationnels positifs $\lambda_1,\ldots,\lambda_n $  avec $\sum_{i=1}^n \lambda_i\leq 1 $,  tels que
 \[\frac{\ve_1}{l}=\sum_{i=1}^n \lambda_{i}  s_i.  \]
 Posons $\beta_i:= \lambda_i l$ pour $i=1,\ldots,n$ et $\beta:=(\beta_1, \ldots,\beta_n) ) $. Alors
\[
\chi^{k \ve_1}=\prod_{i=1}^n\chi^{\beta_i k s_i}= \prod_{i=1}^n t^{k \beta_i b_i} \quad \forall \, t\in
\T(\C),\,\forall\, k\in \N.\\
\]

Comme $\vc_{\overline{D}}$
est invariante par l'action de $S_Q$, soit  $t\in \T(\C)$ et $u$ est l'\'el\'ement de $Q_\R$
v\'erifiant $\exp(-u)=|t|$, alors
{\allowdisplaybreaks
\begin{align*}
\log \bigl\|  \chi^{k \ve_1}\bigr\|_{kl\overline{D}}(t)=& \log \bigl\|  \chi^{k \ve_1}\bigr\|_{kl
\overline{D}}(\exp(-u))\\
=& -\Bigl(\sum_i k\beta_i \bigl<b_i,u\bigr>-kl\, g_{\overline{D}}(u)\Bigr)\\
=& -kl\Bigl(\bigl<\sum_i \frac{\beta_i b_i}{l},u\bigr>-g_{\overline{D}}(u)
\Bigr)\\
=& -kl\Bigl(\bigl< \frac{{}^t B\cdot \beta}{l},u\bigr>-g_{\overline{D}}(u) \Bigr),
\end{align*}}
o\`u on a not\'e par ${}^tB$ la  transpos\'ee de la matrice $B$, dont ses lignes sont $b_1,b_2,\ldots,b_n$. $B$ d\'efinie alors un homomorphisme de $\Z$-modules de $Q$ vers $\Z^d$ associ\'ee au morphisme $\Phi_D$, voir proposition
\eqref{imagetoric}.

On a donc,

 \begin{equation}\label{SupX}
      \bigl\|\chi^{k\ve_1}\bigr\|_{ kl\overline{D},\sup}=\exp\bigl(-kl \check{g}_{\overline{D}}(\frac{{}^t B\cdot \beta}{l})
\bigr).
     \end{equation}

On fixe $\eps>0$. D'apr\`es  $\eqref{gro}$,  il existe une constante $C>0$, telle que
\begin{equation}\label{Gromov11}
\bigl\|\chi^{k\ve_1}\bigr\|_{ kl\overline{D},\sup}\leq C e^{\eps \, k} \bigl\|\chi^{k\ve_1}
\bigr\|_{L^2,  kl\overline{D}, \Omega}\quad \forall k\gg 1.
\end{equation}
Comme on a montr\'e que
\[
\|\chi^{k\ve_1}\|_{L^2,kl \overline{D}, \Omega}\leq 1,\quad \forall k\in \N,
\]
alors,
\[
\exp\bigl(-l \check{g}_{\overline{D}}(\frac{{}^t B\cdot \beta}{l})
\bigr)=
    \bigl\|\chi^{k\ve_1}\bigr\|_{ kl\overline{D},\sup}^\frac{1}{k}\leq C^\frac{1}{k}  e^{\frac{\eps}{2}}\quad \forall
    \,k\in \N_{\geq 1}
     \]                  En faisant tendre $k$ vers l'infini, on obtient:
\[
 \check{g}_{\overline{D}}\bigl(\frac{{}^t B\cdot \beta}{l}\bigr)\geq -\frac{\eps}{2l },
\]
et puisque $\eps$ est arbitraire, alors on  d\'eduit
\begin{equation}\label{x88}
 \check{g}_{\overline{D}}\bigl(\frac{{}^t B\cdot \beta}{l}\bigr)\geq 0,
\end{equation}
 Comme $s_i=({}^t \vf)(e_i^{(n)} ) $ pour
$i=0,\ldots,n$ (o\`u $\{e_1^{(n)},\ldots,e_n^{(n)}\}$ d\'esigne la base standard de $\R^n$)
alors $\textbf{e}_1=l\sum_{i=1}^n \lambda_is_i=({}^t\vf)\bigl(\sum_{i=1}^n l\lambda e_i^{(n)}\bigr)={}^t B\cdot \beta
$ ), c'est \`a dire $\ve_1={}^t B\cdot \beta $. Donc \eqref{x88} devient
\[
\check{g}_{\overline{D}}\bigl(\frac{\e_1}{l}\bigr)\geq 0,
\]
c'est \`a dire
\[
 \ve_1 \in (l\Theta_{\overline{D}})\cap P.
\]
\end{proof}
Soit $\ve_{i_1},\ldots,\ve_{i_q}$  les points extr\'emaux de $\mathrm{Conv}(\ve_1,\ldots,\ve_m)$, alors
\[
\mathrm{Conv}(\ve_{1},\ldots,\ve_{m})= \mathrm{Conv}(\ve_{i_1},\ldots,\ve_{i_q})\subseteq l\Theta_{\overline{D}}
\]


Ce qui termine la preuve de la proposition \eqref{Base}.

\section{Preuve des th\'eor\`emes \eqref{x3} et \eqref{x4}}\label{positari}

\noindent{\large{\sffamily La positivi\'e arithm\'etique}}\\

Nous d\'ecrivons les diff\'erentes notions de positivit\'e arithm\'etique en termes de la combinatoire.

\begin{theorem}\label{amplenefgros} Soit $X$ une vari\'et\'e torique lisse sur $\mathrm{Spec}(\Z)$ et $\overline{D}
=(D,\vc_{\overline{D}})$
un fibr\'e admissible sur $X$ telle que  $\vc_{\overline{D}}$ est
 invariante par l'action du tore compact de $X(\C)$. On a
\begin{enumerate}
\item $\overline{D}$ est ample si et seulement si  $\check{g}_{\overline{D}}(\e)>0 $, $\forall  \e\in \Delta_D\cap P$
et $\psi_D$ est strictement concave.
\item $\overline{D}$ est nef   si et seulement si $\check{g}_{\overline{D}}(\e)\geq 0 $, $\forall \e\in \Delta_D\cap P$ et $\psi_D$ est concave.
\item  $\overline{D} $ est gros si et seulement si $g_{\overline{D}}(0)<0 $.
\end{enumerate}

\end{theorem}
\begin{proof}

\begin{enumerate}
\item Si $\check{g}_{\overline{D}}(\e)>0 $ pour tout $\forall \,\e\in \Delta_D\cap P$, alors $\|\chi^\e\|_{\sup}<1
$, $\forall\,  \e\in \Delta_D\cap P$. Cela implique que  $(l\Delta_D)\cap P\subseteq (l\Theta_{\overline{D}})\cap P
    $ pour tout $l\in \N_{\geq 1}$. Par suite \[\left<\{s\in H^0(X,\mathcal{O}(lD))\,|\,
    \|s\|_{l\overline{D},\sup}<1\}\right>_\Z=H^0(X,\mathcal{O}(lD))\quad \forall\, l\in \N_{\geq 1}, \]
et puisque $\psi_D$ ests strictement concave alors $D$ est ample. On conclut que
 $\overline{D}$ est ample.

R\'eciproquement, supposons que $\overline{D}$ est ample. Alors $\psi_D$ est strictement concave puisque
par hypoth\`ese $D$ est ample et il existe $l\gg 1$, tel que $H^0(X,\mathcal{O}(lD))$ est engendr\'e comme
$\Z$-module, par l'ensemble:
\[
 \bigl\{\phi\in H^0(X,\mathcal{O}(lD))\,|\, \|\phi\|_{\sup}<1   \bigr\}.
\]
Soit $\phi$ un \'el\'ement non nul de cet ensemble. Il existe des entiers $c_{\textbf{e}}$, o\`u $\textbf{e} \in l\Delta_D\cap P$,
tels que:
\[
 \phi=\sum_{\textbf{e}\in l\Delta_D\cap P}c_{\textbf{e}}\chi^{\textbf{e}}.
\]

Soit $\Omega$ une forme de volume de classe $\cl$ et invariante par l'action du tore compact de $X(\C)$ avec
$\int_{X(\C)}\Omega=1$. Comme
 dans la preuve du lemme \eqref{convextreme111}, on pose $\bigl\{\textbf{e}\,|\, c_{\textbf{e}}\neq 0
\bigr\}=\bigl\{\textbf{e}_1,\ldots,\textbf{e}_m \bigr\}$ et on choisit un point extr\'emal de
$\mathrm{Conv}\bigl\{\textbf{e}_1,\ldots,\textbf{e}_m \bigr\}$, qu'on peut supposer \'egal \`a $\textbf{e}_1$.
On a, pour tout $k\in \N_{\geq 1}$,
\[
 \phi^k= c^k_{\ve_1} \chi^{k \ve_1}+ \sum_{ \substack{k_1,\ldots,k_m \in \Z_{\geq 0}\\
k_1+\cdots+k_m=k,k_1\neq k}
 }\frac{k!}{k_1!\cdots k_m !} c_{\ve_1}^{k_1}\cdots c_{\ve_m}^{k_m} \chi^{k_1 \ve_1} \cdots \chi^{k_m
\ve_m}.
\]

On v\'erifie que $k \ve_1\neq k_1 e_1+\cdots + k_m \ve_m  $, pour tout $k_1,\ldots,k_m \in \Z_{\geq 0} $ tels que $k_1+\cdots+k_m=k $ et $k_1\neq k $. Sinon, $\ve_1=(\frac{k_2}{k-k_1})\ve_2 +\cdots+(\frac{k_m}{k-k_1})\ve_m$. Cela contredit le fait que $\ve_1$ est un point extr\'emal de $\mathrm{Conv}(\ve_1,\ldots,\ve_m ) $, par suite on peut \'ecrire

\[
 \phi^k=c^k_{\ve_1} \chi^{k \ve_1}+ \sum_{\ve'\in \Z^d_{\geq 0}, \ve'\neq k\ve_1} c'_{\ve'}\chi^{\ve'},
\]cela implique que
\[
\left<\phi^k,\phi^k\right>_{kl\overline{D},\Omega} =c^{2k}_{\ve_1}\left<\chi^{k \ve_1},\chi^{k \ve_1}\right>_{kl\overline{D},\Omega}+(\text{r\'eel positif }),
\]

(on a v\'erifi\'e  que
$\left<\chi^{\ve_1}\cdots\chi^{\ve_r},\chi^{\ve'_1}\cdots\chi^{\ve'_r}\right>_{kl\overline{D},\Omega}=0,$
si $(\ve_1,\ldots,\ve_r )\neq (\ve'_1,\ldots,\ve'_r )$).\\

On en  d\'eduit que\[
\left<\chi^{k \ve_1},\chi^{k \ve_1}\right>_{kl\overline{D},\Omega}\leq \left<\phi^k,\phi^k\right>_{kl\overline{D},\Omega} ,\quad (\text{rappelons que}\;c^{2k}_{\ve_1}\in \Z).
\]
et donc (on a suppos\'e  $\int_{X(\C)}\Omega=1$)
\[
\left<\chi^{k \ve_1},\chi^{k \ve_1}\right>_{kl\overline{D},\Omega}\leq \left<\phi^k,\phi^k\right>_{kl\overline{D},\Omega}\leq \|\phi^k\|_{kl\overline{D},\sup}^2\leq \|\phi\|_{kl\overline{D},\sup}^{2k}\quad \forall\, k\in
\N_{\geq
1}.
\]
D'apr\`es \eqref{SupX} et \eqref{Gromov11}, on d\'eduit

\[
 \exp\bigl(-l \check{g}_{\overline{D}}(\frac{\textbf{e}_1}{l})
\bigr)\leq C^{\frac{1}{k}}\,e^{\eps} \|\phi\|_{\overline{D},\sup}\quad \forall \, k\in \N_{\geq 1}.
\]
En prenant  $k\mapsto \infty$, on obtient
\[
 l\check{g}_{\overline{D}}(\frac{\textbf{e}_1}{l})\geq -\log \|\phi\|_{\overline{D},\sup}-\eps.
\]
Cette in\'egalit\'e est valable pour tout $\eps>0$. Donc,
\[
 \check{g}_{\overline{D}}(\frac{\textbf{e}_1}{l})\geq -\frac{1}{l}\log\|\phi\|_{\overline{D},\sup}>0 ,
\]
On conclut que, pour tout $c_{\textbf{e}}\neq 0$, on a
\[
\check{g}_{\overline{D}}(\frac{\textbf{e}}{l})>0.
\]

Par hypoth\`ese  $\{\phi\in H^0(X,\mathcal{O}(lD))| \,\|\phi\|_{\overline{D},\sup}<1 \}$  engendre $H^0(X,\mathcal{O}(lD))$. Donc pour tout
$\e\in
l\Delta_D\cap P$,  on peut trouver  $\phi\in \{\phi'\in H^0(X,\mathcal{O}(lD))| \,\|\phi'\|_{\overline{D},\sup}<1 \} $
 tel que $c_{\e}\neq 0$. Cela permet de conclure que
\[
 \check{g}_{\overline{D}}(\textbf{e} )>0\quad \forall\, \e\in \Delta_D\cap P.
\]


\item On suppose que $\check{g}_{\overline{D}}(\e)\geq 0$ pour tout $\e\in \Delta_D\cap P$. Donc, par
concavit\'e de $\check{g}_{\overline{D}}$,
   on a $\check{g}_{\overline{D}}(m)\geq 0$  $\forall \,m\in l\Delta_D\cap P$ et $\forall\, l\in \N_{\geq 1}$. Par
   cons\'equent, pour tout $l\in \N_{\geq 1}$ l'espace $H^0(X,\mathcal{O}(lD))$ est engendr\'e par des sections globales de normes
   sup
inf\'erieure ou \'egale \`a 1. D'apr\'es \cite[proposition 5.5.7]{Maillot}, on  d\'eduit que
$h_{\overline{D}}(P)\geq 0 $ pour tout $P\in X(\overline{\Q})$. On conclut  que $\overline{D}$ est nef.\\

R\'eciproquement si $\overline{D}$ est nef. Donc $h_{\overline{D}}(P)\geq 0 $ pour tout $P\in X(\overline{\Q})$, en particulier pour tout $P\in (X\setminus \mathrm{div}(\chi^{\e}))(\overline{\Q})$
 avec $\e\in \Delta_D\cap P$. Or, si l'on consid\`ere
  $P\in (\Q^\ast)^{d} $ alors on a  $h_{\overline{D}}(P)=-\log \|\chi^{\e} \|_{\overline{D}}(P)$.
   Par cons\'equent, on aura
   $ \|\chi^{\e} \|_{\overline{D}}(P)\leq 1$. Par invariance de la m\'etrique par l'action de $S_Q$ et
    par densit\'e, on d\'eduit que
   $ \|\chi^{\e} \|_{\overline{D}}(x)\leq 1$ pour tout $x\in X(\C)$. Par suite,

    \[
 \check{g}_{\overline{D}}(\textbf{e} )\geq 0\quad \forall\, \e\in \Delta_D\cap P.
\]

\item Si $\overline{D}$ est gros. Par d\'efinition, il existe une section $\phi\in H^0(X,\mathcal{O}(lD))$ pour
un certain $l\gg 1$, avec $\|\phi\|_{\sup}<1$. Alors comme dans  le cas ample, on montre
qu'il existe $\textbf{e}\in l\Delta_D\cap P$ tel que:
\[
 l\check{g}_{\overline{D}}(\frac{\textbf{e}}{l})\geq -2\log \|\phi\|_{\sup}>0,
\]
et comme $-g_{\overline{D}}(0)\geq \min_{u\in \R^d}(<\frac{\textbf{e}}{l},u>-g_{\overline{D}}(u) )=\check{g}_{\overline{D}}(\frac{\textbf{e}}{l}) $  donc,
\[
 g_{\overline{D}}(0)<0.
\]

R\'eciproquement, on suppose que  $g_{\overline{D}}(0)<0$. Comme $g_{\overline{D}}$ est concave alors d'apr\`es
 \cite[p. 218]{convex} il existe $x\in
\R^d$ tel que
\[
 g_{\overline{D}}(u)-g_{\overline{D}}(0)\leq <x, u>\quad \forall u\in \R^d,
\]
On a par suite
\[
   0< -g_{\overline{D}}(0)  \leq \check{g}_{\overline{D}}(x).
     \]
On peut supposer que $x\in \mathrm{Int}(\Delta_D)$. En effet, si l'on consid\`ere $p\in \mathrm{Int}(\Delta_D)$ alors
$tx+(1-t)p\in \mathrm{Int}(\Delta_D)$ d\`es que $t\in [0,1[$ (rappelons que $\Delta_D$ est d\'efini par un nombre fini
d'in\'egalit\'es, voir \eqref{x1}) et par concavit\'e de $\check{g}$, on peut trouver $t\in [0,1[$ tel que
$\check{g}_{\overline{D}}(tx+(1-t)p) <0$.

Comme $ \check{g}_{\overline{D}}$ est concave, alors  on montre que $ \check{g}_{\overline{D}}$ est continue sur
$\mathrm{Int}(\Delta_D)$, voir par exemple \cite[th\'eor\`eme 2.2]{Gruber}. On peut donc supposer que $x\in \Theta_{\overline{D}}\cap \Q^d$. Il existe alors $l$,
 un entier positif non nul tel que $lx\in l\Delta_{D}\cap P$. Si l'on
pose $\e:=lx$, alors d'apr\`es ce qui pr\'ec\`ede  la section globale $\chi^\e$ v\'erifie:
\[
 \|\chi^\e\|_{\overline{D},\sup}=\exp(-l\check{g}_{\overline{D}}(\frac{\e}{l}))=\exp(-l\check{g}_{\overline{D}}(x))<1,
\]
donc, $\overline{D}$ est gros.
\end{enumerate}

\end{proof}

\noindent{\large{\sffamily Le volume arithm\'etique}}\\

Soit $\overline{D}$ un fibr\'e en droites hermitien muni d'une m\'etrique continue sur $X$. Le volume $\widehat{\mathrm{vol}}(\overline{D})$ de $\overline{D}$ est d\'efini comme suit:
\[
\widehat{\mathrm{vol}}(\overline{D})=\limsup_{l\rightarrow \infty}\frac{\log \#\widehat{H}^0(X,l\overline{D}))}{l^{d+1}/(d+1)!}.
\]

C'est l'analogue arithm\'etique du volume d'un fibr\'e en droites sur une vari\'et\'e projective sur un corps.

\begin{example}

\rm{ Si $X$ est une vari\'et\'e torique lisse sur $\mathrm{Spec}(\Z)$ et $\overline{D}_\infty$ est un fibr\'e en droites
\'equivriant sur $X$ muni de sa m\'etrique canonique, alors
\[
 \widehat{\mathrm{vol}}(\overline{D}_\infty)=0.
\]
En effet, d'apr\`es \eqref{JKLP} il suffit de noter qu'on a pour tout $l\gg 1$,
$\#\widehat{H}^0(X,l\overline{D}_\infty)=2\# (l\Delta_D\cap P)+1\simeq 2 \mathrm{vol}(\Delta_D)l^d$.

 }
\end{example}

La suite de cette section est consacr\'ee \`a la preuve du th\'eor\`eme \eqref{x3} (voir \eqref{volumetorpos}). Nous
 adopterons la preuve du
 \cite[th\'eor\`eme 2.3]{Moriwaki} dans le cas torique.

Nous commen\c{c}ons par \'etablir un r\'esultat qui permet d'approximer le volume arithm\'etique par
une quantit\'e plus flexible, d\'efinie en terme
 de la norme $L^2$, c'est l'objet du lemme \eqref{x2}. Notons que ce lemme peut \^etre vue comme
  une version faible du
 \cite[lemme 2.1]{Moriwaki}.
\begin{lemma}\label{x2}
Pour tout $\eps>0$, on a
\[
\widehat{\mathrm{vol}}(\overline{D})\leq
\liminf_{l\mapsto\infty}\frac{\log
\#\widehat{H}^0_{L^2}(X,l\overline{D})}{l^{d+1}/(d+1)!}\leq \limsup_{l\mapsto\infty}\frac{\log
\#\widehat{H}^0_{L^2}(X,l\overline{D})}{l^{d+1}/(d+1)!}\leq \widehat{\mathrm{vol}}(\overline{D}_{\eps}),
\]
avec $\overline{D}_{\eps}:=(D,e^{-\eps}\vc_{\overline{D}})$.
\end{lemma}

\begin{proof}
Puisque  $\widehat{H}^0(X,l\overline{D})\subset \widehat{H}^0_{L^2}(X,l\overline{D})$, alors

 \[
\widehat{\mathrm{vol}}(\overline{D})\leq
\liminf_{l\mapsto\infty}\frac{\log
\#\widehat{H}^0_{L^2}(X,l\overline{D})}{l^{d+1}/(d+1)!}.
 \]
 Soit $\eps>0$, d'apr\`es \eqref{gro}, il existe une constante $C$ tel que $\|\cdot\|_{\sup}\leq C e^{\eps l}
 \|\cdot\|_{L^2}$ sur $H^0(X,\mathcal{O}(lD))$. On peut supposer que $\|\cdot\|_{\sup}\leq  e^{2\eps l} \|\cdot\|_{L^2}$
 pour $l\gg 1$. Cela donne
 \[
\widehat{H}^0_{L^2}(X,l\overline{D})\subset \widehat{H}^0(X,l\overline{D}_{4\eps}) \quad \forall\, l\gg 1.
 \]
 Donc,
 \[
 \limsup_{l\mapsto\infty}\frac{\log
\#\widehat{H}^0_{L^2}(X,l\overline{D})}{l^{d+1}/(d+1)!}\leq \widehat{\mathrm{vol}}(\overline{D}_{4\eps}).
 \]
\end{proof}
\begin{remarque}
\rm{Dans \cite[lemme 2.1]{Moriwaki}, Moriwaki a montr\'e une in\'egalit\'e
semblable \`a celle du lemme pr\'ec\'edent, et par contuit\'e de $\widehat{\mathrm{vol}}$, il a pu  conclure que $\widehat{\mathrm{vol}}(\overline{D})=\lim_{l\mapsto \infty
}\frac{\log
\#\widehat{H}^0_{L^2}(X,l\overline{D})}{l^{d+1}/(d+1)!} $. Malheureusement, la continuit\'e de $\widehat{\mathrm{vol}}$
est \'etablie uniquement dans le cas des m\'etriques de classe $\cl$ (voir \cite[th\'eor\`eme B]{Moriwaki2}).}

\end{remarque}

\begin{lemma}\label{key}
Soit $\Theta$ un sous-ensemble convexe compact de $\R^d$ tel que $\mathrm{Vol}(\Theta)>0$ ($\mathrm{Vol}$
d\'esigne le volume induit par la mesure de Lebesgue $dx$ standard de $\R^d$). Pour tour $l$ un entier positif non nul, soit $A_l=(a_{\textbf{e},\textbf{e}'})_{\e,\e'\in l\Theta\cap P}$ une matrice r\'eelle sym\'etrique d\'efinie positive index\'ee par $l\Theta\cap P$, et soit $K_l$ le sous-ensemble de $\R^{l\Theta\cap P}\simeq \R^{\#(l\Theta\cap P)}$ donn\'e par
\[
 K_l=\{(x_\e)\in \R^{l\Theta\cap P}\,|\, \sum_{\e,\e'\in l\Theta\cap P}a_{\e,\e'}x_\e x_{\e'}\leq 1 \}.
\]
On suppose qu'il existe une fonction continue $\vf:\Theta\longrightarrow \R$ tel que pour   tout $\eps>0$, il existe une constante $D$  v\'erifiant
\[
 |\log(\frac{1}{a_{\e,\e}})-l\vf(\frac{\e}{l})|\leq D+\eps\, l,
\]
 pour tout $l$ un entier positif assez grand et  $\e\in l\Theta\cap P$. Alors on a
\[
\lim_{l\mapsto \infty}\inf\frac{\log \#(K_l\cap\Z^{l\Theta\cap P} )}{l^{d+1}}\geq \frac{1}{2}\int_{\Theta}\vf(x)dx
\]
En plus, si $A_l$ est diagonal et $a_{\e,\e'}\leq 1$ $\forall \e,\e'\in l\Theta\cap P$ pour tout $l$, on a
\[
 \lim_{l\mapsto\infty}\frac{\log \#(K_l\cap\Z^{l\Theta\cap P} )}{l^{d+1}}= \frac{1}{2}\int_{\Theta}\vf(x)dx.
\]

\end{lemma}
\begin{proof}
La preuve de ce lemme est une l\'eg\`ere modification de la preuve du \cite[Lemme 2.2]{Moriwaki}. Soit $\eps>0$, par hypoth\`ese il
existe une constante $D$ telle que
\[
 |\log(\frac{1}{a_{\e,\e}})-l\vf(\frac{\e}{l})|\leq D+\eps\, l\quad \forall \, l\gg 1\quad \e\in l\Theta\cap P.
\]
D'apr\`es \cite[p. 514]{Moriwaki}, on a
\begin{equation}\label{Mink}
\log \#(K_l\cap \Z^{l\Theta\cap P})\geq \frac{1}{2}\sum_{\e \in l\Theta\cap P}\log \bigl(\frac{1}{a_{\e,\e}}\bigr)
+\log V_{m_l}-m_l \log 2,
\end{equation}
o\`u $m_l=\#(l\Theta\cap P)$, et $V_{m_l}$ est le volume de la boule unit\'e dans $\R^{m_l}$ muni de la m\'etrique
standard. \\

Par hypoth\`ese,
\[
\vf(\frac{\e}{l}) -\frac{1}{l}D-\eps \leq \frac{1}{l}\log(\frac{1}{a_{\e,\e}})\leq \vf(\frac{\e}{l})+ \frac{1}{l}D+\eps \quad \forall \, l\gg 1\quad \forall\,\e\in l\Theta\cap P.
\]
Donc,
\[
\frac{1}{l^d}\sum_{\e\in l\Theta\cap P}\vf(\frac{\e}{l}) -\frac{m_l}{l^{d+1}}D-\frac{m_l}{l^d}\eps \leq \frac{1}{l^{d+1}}\sum_{\e\in l\Theta\cap P}\log(\frac{1}{a_{\e,\e}})\leq\frac{1}{l^d}\sum_{\e\in l\Theta\cap P}\vf(\frac{\e}{l}) +\frac{m_l}{l^{d+1}}D+\frac{m_l}{l^d}\eps  \quad \forall \, l\gg 1.
\]
Notons que
\[
\lim_{l\mapsto \infty} \frac{1}{l^d}\sum_{\e\in l\Theta\cap P}\vf(\frac{\e}{l})=\lim_{l\mapsto \infty} \sum_{x\in
\Theta\cap (1/l)P}\vf(x)=\int_\Theta \vf(x)dx.
\] Alors, on peut trouver $l_0\gg 1$ tel que
\[
\bigl| \frac{1}{l^{d+1}}\sum_{\e\in l\Theta\cap P}\log(\frac{1}{a_{\e,\e}})-\int_\Theta \vf(x)dx \bigr|\leq \eps
+\frac{m_l}{l^{d+1}}D+\frac{m_l}{l^d}\eps\quad \, \forall\, l\geq l_0.
\]
Comme il existe une constante $c$ telle que $m_l\leq l^d$. Alors, on peut supposer que,
\[
\bigl| \frac{1}{l^{d+1}}\sum_{\e\in l\Theta\cap P}\log(\frac{1}{a_{\e,\e}})-\int_\Theta \vf(x)dx \bigr|\leq
(2+c_1)\eps\quad \, \forall\, l\geq l_0.
\]
Par suite,
\[
\lim_{l\mapsto \infty}\frac{1}{l^{d+1}}\sum_{\e\in l\Theta\cap P}\log(\frac{1}{a_{\e,\e}})=\int_\Theta \vf(x)dx.
\]
En utilisant \eqref{Mink}, on d\'eduit de ce qui pr\'ec\`ede
\[
\liminf_{l\mapsto \infty}\frac{1}{l^{d+1}}\log \#(K_l\cap \Z^{l\Theta_{\overline{D}}\cap P})\geq \frac{1}{2}\int_\Theta \vf(x)dx.
\]
La preuve de la deuxi\`eme assertion est identique \`a celle de la deuxi\`eme partie du \cite[lemme 2.2]{Moriwaki}.
\end{proof}

\begin{theorem}\label{volumetorpos}
 Soit $X$ une vari\'et\'e torique lisse. Soit $\overline{D}=(D,\vc_{\overline{D}})$ un fibr\'e en
droites  \'equivariant muni d'une m\'etrique    $\vc_{\overline{D}}$, continue  et
 invariante par l'action de tore compact de $X(\C)$. On a,
\[
 \widehat{\mathrm{vol}}(\overline{D})=(d+1)!\int_{\Theta_{\overline{D}}}\check{g}_{\overline{D}}(x)dx.
\]

\end{theorem}

\begin{proof}


On consid\`ere $\Omega$ une forme volume  de classe $\cl$ sur $X(\C)$,   invariante par l'action du tore
compact de $X(\C)$ et telle que $\int_{X(\C)}\Omega=1$.
 Soit $l$ un entier positif non nul. Posons $A_l=(a_{\e,\e'})_{\e,\e'\in l\Theta\cap P}$ avec
$a_{\e,\e'}=\left<\chi^\e,\chi^{\e'}\right>_{l\overline{D},\Omega}$. On a, $\Theta_{\overline{D}}$ est un ensemble
 compact et convexe. On   consid\'ere la fonction
$\check{g}_{\overline{D}}:\Theta_{\overline{D}}\longrightarrow \R$ et $K_l$ l'ensemble donn\'e par
\[
 K_{l}=\{(x_\e)\in \R^{l\Theta\cap P}\,|\, \sum_{\e,\e'\in l\Theta\cap P}a_{\e,\e'}x_\e x_{\e'}\leq
1 \}.
\]

Soit $e\in l\Theta_{\overline{D}}\cap P$ et   $\chi^{\textbf{e}}$ la  section globale de $\mathcal{O}(lD)$ associ\'ee. D'apr\`es
la proposition \eqref{gro}, on a pour tout $\eps>0$, il existe une constante $C$ telle que
\[
 \| \chi^{\textbf{e}}\|_{l {\overline{D}},\sup}\leq C\, e^{\eps \,l}  \|\chi^{\textbf{e}}\|_{L^2,l{\overline{D}}}\quad
 \forall\, l\gg 1.
\]
On a aussi,
\[
 \|\chi^{\textbf{e}}\|_{L^2,l\overline{D}}\leq \|\chi^{\textbf{e}}\|_{l\overline{D},\sup}
\int_{X}\Omega=\|\chi^{\textbf{e}}\|_{l \overline{D},\sup}
\]

Avec les notations introduites et la formule \eqref{SupX}, les deux in\'egalit\'es pr\'ec\'edentes deviennent:
\[
 \exp(-l\, \check{g}_{\overline{D}}(\frac{\textbf{e}}{l}))\leq C  e^{\eps\, l}\sqrt{a_{\textbf{e},\textbf{e}}},
\]
et
\[
 \sqrt{a_{\textbf{e},\textbf{e}}}\leq
 \exp(-l\, \check{g}_{\overline{D}}(\frac{\textbf{e}}{l})).
\]
On en d\'eduit
\[
0\leq \log(\frac{1}{a_{\textbf{e},\textbf{e}}})-2l \check{g}_{\overline{D}}(\frac{\textbf{e}}{l})\leq 2\log C+\eps\, l\quad \forall\,
l\gg 1.
\]
Puisque $\textbf{e} \in l\Theta_{\overline{D}}\cap P$, et  $\sqrt{a_{\textbf{e},\textbf{e}}}\leq
 \exp(-l\, \check{g}(\frac{\textbf{e}}{l}))$ alors \[
 a_{\textbf{e},\textbf{e}}\leq 1.
\]
Comme $\vc_{\overline{D}}$ est invariante par l'action du tore compact, alors la matrice $A_l$ est diagonale. Maintenant on peut
 appliquer le lemme \eqref{key} et nous obtenons
\[
\lim_{l\mapsto\infty}\frac{\log \#(K_l\cap\Z^{l\Theta_{\overline{D}}\cap P} )}{l^{d+1}}= \int_{\Theta_{\overline{D}}}\check{g}_{\overline{D}}(x)dx.
\]

En remarquant  que
\[
\widehat{H}^0_{L^2}(X,l\overline{D})= K_l\cap\Z^{l\Theta_{\overline{D}}\cap P},
\]
alors
\[
\liminf_{l\mapsto\infty}\frac{\log
\#\widehat{H}^0_{L^2}(X,l\overline{D})}{l^{d+1}/(d+1)!}=\limsup_{l\mapsto\infty}\frac{\log
\#\widehat{H}^0_{L^2}(X,l\overline{D})}{l^{d+1}/(d+1)!}=\int_{\Theta_{\overline{D}}}\check{g}_{\overline{D}}(x)dx.
\]
Et d'apr\`es le lemme \eqref{x2}, on a pour  tout $\eps>0$
\[
\widehat{\mathrm{vol}}(\overline{D})\leq
(d+1)!
\int_{\Theta_{\overline{D}}}\check{g}_{\overline{D}}(x)dx\leq \widehat{\mathrm{vol}}(\overline{D}_{\eps}).
\]
Si l'on remplace $\overline{D}$ par $\overline{D}_{-\eps}$, on obtient que
\[
(d+1)!
\int_{\Theta_{\overline{D}_{-\eps}}}\check{g}_{\overline{D}_{-\eps}}(x)dx\leq
\widehat{\mathrm{vol}}(\overline{D}).
\]
Or, $\check{g}_{\overline{D}_{-\eps}}(x)=\check{g}_{\overline{D}}(x)-2\eps$  et $\Theta_{\overline{D}_{-\eps}}=
\{x\in \Delta_D|\,\check{g}_{\overline{D}}(x)\geq 2\eps   \}$. Par cons\'equent,
$\int_{\Theta_{\overline{D}_{-\eps}}}\check{g}_{\overline{D}_{-\eps}}(x)dx=
\int_{\Theta_{\overline{D}}}\check{g}_{\overline{D}}(x)dx+ \int_{\Theta_{\overline{D}}
\setminus \Theta_{\overline{D}_{-\eps}}}\check{g}_{\overline{D}}(x)dx-2\eps \mathrm{\Theta_{\overline{D}_{-\eps}}}=
\int_{\Theta_{\overline{D}}}\check{g}_{\overline{D}}(x)dx+O(\eps)$. En faisant tendre $\eps$ vers 0, on obtient

\[
(d+1)!
\int_{\Theta_{\overline{D}}}\check{g}_{\overline{D}}(x)dx\leq
\widehat{\mathrm{vol}}(\overline{D}).
\]
On conclut que,

\[
 \widehat{\mathrm{vol}}(\overline{D})=(d+1)!\int_{\Theta_{\overline{D}}}\check{g}_{\overline{D}}(x)dx.
\]


\end{proof}

\bibliographystyle{plain}
\bibliography{biblio}

\begin{thebibliography}{10}

\bibitem{Batyrev}
Victor~V. Batyrev and Yuri Tschinkel.
\newblock Rational points of bounded height on compactifications of anisotropic
  tori.
\newblock {\em Internat. Math. Res. Notices}, (12):591--635, 1995.

\bibitem{Burgos3}
Jos{\'e}~Ignacio Burgos~Gil, Atsushi Moriwaki, Patrice Philippon, and
  Mart{\'{\i}}n Sombra.
\newblock {A}rithmetic positivity on toric varieties.
\newblock {\em arXiv.org}, arXiv:1210.7692 [math.AG], October 2012.

\bibitem{Burgos2}
Jos{\'e}~Ignacio Burgos~Gil, Patrice Philippon, and Mart{\'{\i}}n Sombra.
\newblock {A}rithmetic geometry of toric varieties. {M}etrics, measures and
  heights.
\newblock {\em arXiv.org}, arXiv:1105.5584v1 [math.AG], Mai 2011.

\bibitem{Demazure}
Michel Demazure.
\newblock Sous-groupes alg\'ebriques de rang maximum du groupe de {C}remona.
\newblock {\em Ann. Sci. \'Ecole Norm. Sup. (4)}, 3:507--588, 1970.

\bibitem{Fulton}
William Fulton.
\newblock {\em Introduction to toric varieties}, volume 131 of {\em Annals of
  Mathematics Studies}.
\newblock Princeton University Press, Princeton, NJ, 1993.
\newblock The William H. Roever Lectures in Geometry.

\bibitem{GKZ}
I.~M. Gelfand, M.~M. Kapranov, and A.~V. Zelevinsky.
\newblock {\em Discriminants, resultants and multidimensional determinants}.
\newblock Modern Birkh\"auser Classics. Birkh\"auser Boston Inc., Boston, MA,
  2008.
\newblock Reprint of the 1994 edition.

\bibitem{Amplitude}
Henri Gillet and Christophe Soul{\'e}.
\newblock Amplitude arithm\'etique.
\newblock {\em C. R. Acad. Sci. Paris S\'er. I Math.}, 307(17):887--890, 1988.

\bibitem{Gruber}
Peter~M. Gruber.
\newblock {\em Convex and discrete geometry}, volume 336 of {\em Grundlehren
  der Mathematischen Wissenschaften [Fundamental Principles of Mathematical
  Sciences]}.
\newblock Springer, Berlin, 2007.

\bibitem{Maillot}
Vincent Maillot.
\newblock {G}\'eom\'etrie d'{A}rakelov des vari\'et\'es toriques et fibr\'es en
  droites int\'egrables.
\newblock {\em M\'em. Soc. Math. Fr. (N.S.)}, 80:vi+129, 2000.

\bibitem{Moriwaki2}
Atsushi Moriwaki.
\newblock Continuity of volumes on arithmetic varieties.
\newblock {\em J. Algebraic Geom.}, 18(3):407--457, 2009.

\bibitem{Moriwaki}
Atsushi Moriwaki.
\newblock Big arithmetic divisors on the projective spaces over {$\Bbb Z$}.
\newblock {\em Kyoto J. Math.}, 51(3):503--534, 2011.

\bibitem{Oda}
Tadao Oda.
\newblock Convex bodies and algebraic geometry---toric varieties and
  applications. {I}.
\newblock In {\em Algebraic {G}eometry {S}eminar ({S}ingapore, 1987)}, pages
  89--94. World Sci. Publishing, Singapore, 1988.

\bibitem{PS}
Patrice Philippon and Mart{\'{\i}}n Sombra.
\newblock Hauteur normalis\'ee des vari\'et\'es toriques projectives.
\newblock {\em J. Inst. Math. Jussieu}, 7(2):327--373, 2008.

\bibitem{convex}
R.~Tyrrell Rockafellar.
\newblock {\em Convex analysis}.
\newblock Princeton Landmarks in Mathematics. Princeton University Press,
  Princeton, NJ, 1997.
\newblock Reprint of the 1970 original, Princeton Paperbacks.

\bibitem{Zhang}
Shouwu Zhang.
\newblock Small points and adelic metrics.
\newblock {\em J. Algebraic Geom.}, 4(2):281--300, 1995.

\end{thebibliography}

\vspace{1cm}

\begin{center}
{\sffamily \noindent National Center for Theoretical Sciences, (Taipei Office)\\
 National Taiwan University, Taipei 106, Taiwan}\\

 {e-mail}: {hajli@math.jussieu.fr,\, hajlimounir@gmail.com}

\end{center}

\end{document}